\documentclass[11pt,a4paper]{article}

\RequirePackage[OT1]{fontenc}
\usepackage{amsfonts, amsmath, amssymb, amsthm,url,dsfont, gensymb,multicol}
\usepackage[margin=3cm,footskip=1cm]{geometry}
\usepackage{graphicx,xcolor}
\usepackage[english]{babel}
\RequirePackage[colorlinks,citecolor=blue,urlcolor=blue]{hyperref}
\usepackage{booktabs}

\RequirePackage[numbers]{natbib}

\usepackage[above,below,section]{placeins}

\usepackage{authblk}
\usepackage{xspace}

\makeatletter
\renewenvironment{abstract}{%
  \small%
  \providecommand\keywords{%
    \par\medskip\noindent\textit{Keywords:}\xspace}%
  \begin{center}%
    \bfseries \abstractname\vspace{-.5em}\vspace{\z@}%
  \end{center}%
  \quote%
}{\endquote}
\makeatother

\usepackage[shortlabels]{enumitem}
\setlist{
  listparindent=\parindent,
  parsep=0pt,
}

\newtheorem{thm}{Theorem}[section]
\newtheorem{cor}[thm]{Corollary}
\newtheorem{lem}[thm]{Lemma}

\newtheorem{prop}[thm]{Proposition}
\theoremstyle{definition}
\theoremstyle{remark} 

\DeclareMathOperator{\Cov}{Cov}
\DeclareMathOperator{\Var}{\mathsf{Var}}

\newcommand{\dd}{{\mathrm d}}

\newcommand{\E}{\mathbb{E}}

\newcommand{\PP}{\mathcal{P}}

\newcommand{\R}{\mathbb{R}}

\newcommand{\one}{\mathds{1}}

\newcommand{\bfx}{{\mathbf{x}}}
\newcommand{\eps}{\varepsilon}

\def\d{{\mathrm{d}}}

\newcommand{\T}{{\sf T}}

\date{}

\providecommand{\keywords}[1]
{
	\small	
	{\noindent\textit{Keywords: }} #1
}

\providecommand{\ams}[1]
{
	\small	
	{\noindent\textit{2020 Mathematics Subject Classification: }} #1
}

\begin{document}

	\title{A functional central limit theorem for the $K$-function
          with an estimated intensity function}

	\author{Anne Marie Svane, Christophe Biscio, Rasmus Waagepetersen}

\maketitle
	\begin{abstract}
          The $K$-function is arguably the most important functional summary
          statistic for spatial point processes. It is used
          extensively for goodness-of-fit testing and in connection
          with minimum contrast estimation for parametric spatial
          point process models. It is thus pertinent to understand
          the asymptotic properties of estimates of the
          $K$-function. In this paper we derive the functional
          asymptotic distribution for the $K$-function
          estimator. Contrary to previous papers on functional
          convergence we consider the case of an inhomogeneous
          intensity function. We moreover handle the fact that
          practical $K$-function estimators rely on plugging in an
          estimate of the intensity function. This removes two serious
          limitations of the existing literature.
	\end{abstract}


\ams{60F17; 60G55; 60F05}

\section{Introduction}	

Ripley's $K$-function \citep{ripley:77} is the most popular summary of
the second moment structure of a spatial point process. Asymptotic properties of the $K$-function are of interest
when the $K$-function is used for goodness-of-fit test
\cite{heinrich:91} of a proposed
point process model or for deriving asymptotic properties of
parameter estimates obtained from minimum contrast estimating
functions based on the $K$-function
\cite{heinrich:92,guan:sherman:07,waagepetersen:guan:09}.

A complete characterization of the
functional asymptotic distribution of the $K$-function for wide
classes of stationary Gibbs point processes and stationary so-called
conditionally $m$-dependent point processes was obtained in \cite{biscio:svane:22}. The results were derived
assuming known intensity and hence applicable for
goodness-of-fit testing for a specific type of model with a given
intensity. It is, however, in practice often desired to test the goodness-of-fit
of a given type of model considering the intensity an unknown
parameter to be estimated. Then the asymptotic distribution of the
$K$-function needs to be adjusted for the effect of replacing the true
intensity by an estimate.  Functional
convergence of statistics related to the $K$-function for both known and unknown
intensity was considered in \cite{heinrich:91}. However, the approach in this paper relied heavily on the
setting of a stationary Poisson process.

Assuming a constant intensity is often restrictive and \cite{baddeley:moeller:waagepetersen:00} extended the $K$-function to
a wide class of point processes with inhomogeneous intensity
functions. This allowed to study residual second order structure after
filtering out large scale variations due to a non-constant intensity
function. In practice an estimate of the intensity function is plugged
in for the typically unknown intensity function. Parametric
estimates can be obtained using for example composite
likelihood or quasi-likelihood, e.g.\
\cite{schoenberg:05,waagepetersen:07,guan:loh:07,waagepetersen:guan:09,guan:jalilian:waagepetersen:15}. Consistency
and asymptotic normality of parameter estimates is studied in detail
in these papers. For example, \cite{waagepetersen:guan:09} used a framework of non-stationary
$\alpha$-mixing spatial point processes  and established joint asymptotic 
normality of composite likelihood estimates for the intensity and
minimum-contrast estimates of clustering parameters where the contrast
was based on the $K$-function with estimated intensity. However,
\cite{waagepetersen:guan:09} did not consider functional convergence.

In this paper we address limitations of the existing literature by
establishing a functional central limit theorem for the
$K$-function in the case of an unknown possibly non-constant intensity
function. Our main focus is establishing
functional convergence when the intensity function is replaced by a
parametric estimate. We assume
as our starting point the
availability of a multivariate central limit theorem for a random vector
consisting of intensity function parameter estimates and estimates of
the $K$-function with known intensity function for a range of spatial
distances. This assumption is justified by the aforementioned references. 

\section{The $K$-function} \label{sec:K}

Let $\PP\subseteq \R^d$ be a simple point process with intensity function
$\rho(\cdot)>0$ and let $A\subseteq \R^d$ be a set of  positive and finite
volume $|A|$. Assuming further that  $\PP$ is second-order intensity
reweighted stationary \citep{baddeley:moeller:waagepetersen:00}, the
$K$-function is defined for any $r>0$ as 
\begin{equation*}
  K(r) = \frac{1}{|A|}\E \sum_{x\in \PP\cap A}
  \sum_{y\in \PP } \frac{\mathds{1}_{\{0<\|x-y\|\le r\}}}{\rho(x)
    \rho(y)},
\end{equation*}
where the right hand side does not depend on the particular choice
of $A$. The $K$-function can also be expressed as a Palm expection
\[ 
  K(r)= \E_u  \sum_{y\in \PP}\frac{\mathds{1}_{\{0<\|y-u\|\le r\}}}{\rho(y)},
\]
for any $u \in \R^d$, where $\E_u$ denotes the Palm expectation given
$u\in \PP$. These definitions agree with the definition of Ripley's
$K$-function in the stationary case where $\rho(\cdot)$ is constant.

In practice $\PP$ is  only observed inside a bounded
observation window. Throughout this paper, we will consider a
square observation window  $W_n =
[-\frac{1}{2}n^{1/d},\frac{1}{2}n^{1/d}]^d$ of volume $n$ and
write $\PP_n = \PP \cap W_n$. Assuming for a moment that the intensity
function is known, an unbiased estimator of $K(r)$ is given by
\begin{equation}\label{eq:Kestknown}
  \hat{K}_n(r) =  \sum_{x\in \PP_n}
\sum_{y\in \PP_n}
\frac{\mathds{1}_{\{0<\|x-y\|\le r\}}}{\rho(x)\rho(y)} e_n(x,y),
\end{equation}
where $e_n(x,y)$ is an edge correction factor ensuring unbiasedness. A
popular choice is $e_n(x,y)= |W_n \cap W_{n,x-y}|^{-1}$ where
$W_{n,x-y}$ is $W_n$ translated by $x-y$.

In practice $\rho(\cdot)$ is unknown and must be estimated. We assume
that $\rho(\cdot)$ belongs to a parametric model $\rho_\beta(\cdot)$, $\beta \in
\R^p$, so that $\rho(\cdot)=\rho_{\beta^*}(\cdot)$ for some fixed
parameter value $\beta^* \in \R^d$. A common example is the log linear model
$\rho_\beta(\cdot)=\exp( z(u)^\T \beta)$ where for $u \in \R^d$,
$z(u)$ is a $p$-dimensional covariate vector assumed to be observed
within $W_n$. Methods for obtaining consistent and
asymptotically normal estimates of $\beta$ include composite
likelihood and quasi-likelihood, e.g.\
\cite{schoenberg:05,waagepetersen:07,guan:loh:07,waagepetersen:guan:09,guan:jalilian:waagepetersen:15,choiruddin:coeurjolly:waagepetersen:21}. We
denote by $\hat \beta_n$ an estimate of $\beta$ obtained from an
observation of $\PP \cap W_n$. Define
\[   \hat{K}_{n,\beta}(r) =  \sum_{x\in \PP_n}
\sum_{y\in \PP_n}
\frac{\mathds{1}_{\{0<\|x-y\|\le r\}}}{\rho_{\beta}(x)\rho_{\beta}(y)} e_n(x,y).
\]
Then $\hat K_n(r)=\hat K_{n,\beta^*}(r)$. In practice we estimate
$K(r)$ by $\hat K_{n,\hat \beta_n}(r)$. 

We will assume throughout that
\begin{equation}\label{eq:consistency}
\hat \beta_n \text{ and } \hat K_n \text{ are consistent. }
\end{equation}
Moreover, we will often make the following assumption on $\rho(\cdot)$:
  There is an $\eps>0$ and constants $c_1,\ldots,c_4>0$ such that  for all $(\beta,x)\in
	B_\eps(\beta^*)\times \R^{d}$, 
	\begin{equation}\label{eq:rhoassumption}
	c_1<\rho_\beta(x)<c_2,\qquad \left\|\frac{\dd}{\dd
		\beta}\rho_{\beta}(x)\right\|<c_3, \qquad \left\|\frac{\dd^2}{\dd \beta^\T \dd
		\beta}\rho_{\beta}(x)\right\|<c_4.
	\end{equation}
Here $B_{r}(x)$ denotes the Euclidean ball in $\R^d$ of radius $r$
centered at $x$.

We assume existence of joint intensity functions $\rho^{(l)}$,
$l=2,3,4$ \cite[e.g.\ ][]{rasmus} and define normalized joint intensities
$g^{(l)}(u_1,\ldots,u_l)=\rho^{(l)}(u_1,\ldots,u_l)/\prod_{i=1}^l\rho(u_i)$,
  $l=2,3,4$. In particular, $g^{(2)}$ is the pair correlation function, which will simply be denoted by $g$. Assuming that $g$ is translation-invariant,
  $g(u,v)=g(u-v)$, the $K$-function
  can be written as $K(r)=\int_{B_r(0)}g(h) \dd h$. We will also assume translation invariance of $g^{(3)}$ and
$g^{(4)}$. We note that there are wide classes of point processes
for which translation invariant normalized joint intensities exist
such as log
Gaussian Cox processes \cite{moeller:syversveen:waagepetersen:98},
inhomogeneous Neyman-Scott processes \cite{waagepetersen:07}, and
determinantal point processes \cite{lavancier:moeller:rubak:15}. 

We say that $\PP$ hast fast decaying correlations if for any
$p,q \ge 1$ with $p+q \le 4$, there exists a function $\phi_{p,q}:[0,\infty[
\rightarrow [0,\infty[$ such that $\int_{0}^\infty
\phi_{p,q}(r)r^{d-1} \dd r < \infty$ and 
\begin{equation}\label{eq:fastdecay} | g^{(p+q)}(\bfx,\bfx')-g^{(p)}(\bfx))g^{(q)}(\bfx')| \le
  \phi_{p,q}(d(\bfx,\bfx')) 
  \end{equation} 
where $\bfx$ and $\bfx'$ are point configurations of cardinality $p$
and $q$, respectively, and  $d(\bfx,\bfx')=\min_{u \in \bfx,v \in \bfx'}\|u-v\|$, and we
define $g^{(1)}(u)=1$, $u \in \R^d$. If
$\PP$ has fast decaying correlations \eqref{eq:fastdecay} it follows easily that
\begin{equation}\label{eq:decayg}
 \int_{\R^d}|g(h)-1| \dd h < \infty,
\end{equation}
\begin{equation}\label{eq:decayg3}
\sup_{x\in\R^d}\int_{\R^d}|g^{(3)}(0,x,z)-g(x)| \dd z < \infty
\end{equation}
and
\begin{equation}\label{eq:decayg4}
\sup_{u_1,u_2\in\R^d}\int_{\R^d}|g^{(4)}(0,u_1,u_4,u_2+u_4)-g(u_1)g(u_2)| \dd u_4 < \infty.
\end{equation}
We finally need a condition of bounded normalized densities:
\begin{equation}\label{eq:boundedg}  g^{(k)} \text{ are bounded for } k=2,3,4.\end{equation}

\section{Central limit theorem with estimated intensity function}\label{sec:clt}

In this section we demonstrate how central limit theorems for $K$-functions with unknown intensity can be deduced from analoguous results with known intensity under suitable assumptions.  We postpone a discussion of point process types satisfying the various assumptions to Section \ref{sec:pp}. 

\subsection{Finite dimensional distributions}
Consider for any $k \ge 1$ and $0 < r_1<r_2<\cdots<r_k< \infty$, the
vector 
\[ U_n=((\hat \beta_n- \beta^*)^\T,\hat
K_n(r_1)-K(r_1),\ldots,\hat
K_n(r_k)-K(r_k))^\T.\] Throughout this paper we assume
 asymptotic normality,
\begin{equation}\label{eq:clt}
|W_n|^{1/2}C_{n}U_n \rightarrow N(0,I_{p+k}), 
\end{equation}
for a sequence of matrices $C_n$ where $C_n^{-1}(C_n^{-1})^{\T}/|W_n|$
approximates $\Var U_n$. We discuss these assumptions in  more detail
in Section~\ref{sec:pp}.

The objective in this section is to obtain a central limit theorem for 
\[
V_n=((\hat \beta_n-\beta^*)^\T,\hat K_{n,\hat \beta_n}(r_1)-K
(r_1),\ldots,\hat K_{n,\hat \beta_n}(r_k)-K
(r_k))^\T\] 
as well as consistency of $\hat K_{n,\hat \beta_n}(r)$ for $r \ge 0$.
For this we employ a first order Taylor expansion to obtain
\begin{equation}\label{eq:KTaylor} \hat K_{n,\hat \beta_n}-K(r) = H_{n,\tilde \beta_{n,r}}(r) (\hat
  \beta_n-\beta^*)+ \hat K_{n}(r)-K(r), \end{equation}
where $\|\tilde \beta_{n,r}- \beta^*\| \le \|\hat \beta_{n}- \beta^* \|$ and
\begin{equation}\label{eq:Hn}
H_{n,\beta}(r)= -\sum_{x\in \PP_n}
\sum_{y\in \PP_n}
\frac{\mathds{1}_{\{0<\|x-y\|\le
    r\}}}{\rho_{\beta}(x)\rho_{\beta}(y)} \frac{\dd}{\dd \beta^\T}\log[\rho_{\beta}(x)\rho_{\beta}(y)]
e_n(x,y).
\end{equation}

Let $\tilde B_n$ denote the matrix with columns $\tilde \beta_{n,r_i}$
and let $H_n(\tilde B_n)$ denote the $k \times p$ matrix with rows $H_{n,\tilde \beta_{n,r_i}}(r_i)$,
$i=1,\ldots,k$. Further let
\[ A_n= \begin{bmatrix}I_p & 0_{p \times k}\\ H_n(\tilde B_n) & I_k\end{bmatrix}\]
where $0_{p\times k}$ is a $p \times k$ matrix of zeros. Then,
since $A_n$ is invertible,
\[ |W_n|^{1/2}C_n A_n^{-1} V_n = |W_n|^{1/2}C_{n} U_n .\]
Thus, by \eqref{eq:clt}, $|W_n|^{1/2}C_n A_n^{-1} V_n$ is
asymptotically $N(0,I_{p+k})$. 
This can be used to 
  construct a joint confidence ellipsoid for
  $((\beta^*)^\T,K(r_1),\ldots,K(r_k))^\T$.

We can estimate $H_{n,\tilde \beta_{n,r}}(r)$ by
$H_{n,\hat \beta_{n,r}}(r)$ according to the following
proposition which is also useful for establishing consistency of
$\hat K_{n,\hat \beta_n}$.

\begin{prop}\label{prop:H_est}
 Assume that \eqref{eq:consistency}, \eqref{eq:rhoassumption} and \eqref{eq:fastdecay} are satisfied. Then
 $H_{n,\tilde \beta_{n,r}}(r)-H_{n,\beta^*}(r)$ and 
 $H_{n,\tilde \beta_{n,r}}(r)-H_{n,\hat \beta_{n,r}}(r)$   converge to zero in
 probability.  Further, $\bar H_{n}(r)=\E H_{n,\beta^*}(r)$ is
   bounded. Finally, 
 $\Var H_{n,\beta^*}(r)$ is
   $O(n^{-1})$. 
\end{prop}

\begin{proof}
Define $h_{x,y}(\beta)=\rho_{\beta}(x)\rho_{\beta}(y)
\frac{\dd}{\dd \beta^\T}\log[\rho_{\beta}(x)\rho_{\beta}(y)]$ and let
$h_{i,x,y}(\beta)$ be the $i$th component of $h_{x,y}$. Then
\[ h_{i,x,y}(\tilde \beta_{n,r})-h_{i,x,y}(\beta^*)= h'_{i,x,y}(\bar
  \beta_{i,x,y,n})(\tilde \beta_{n,r}-\beta^*)\]
with $h'_{i,x,y}(\beta)= \frac{\dd}{\dd \beta^\T}
h_{i,x,y}(\beta)$ and $\|\bar \beta_{i,x,y,n}-\beta^*\| \le
\|\tilde{\beta}_{n,r}-\beta^*\| $. Next, $h_{x,y}(\tilde \beta_{n,r})-h_{x,y}(\beta^*)= (\tilde \beta_{n,r}-\beta^*)^\T h'_{x,y}(\bar
B_{x,y,n})^\T$ where $h'_{x,y}(\bar
B_{x,y,n})$ is $p \times p$ with rows $h'_{i,x,y}(\bar
  \beta_{i,x,y,n})$.
Further,
\begin{equation}\label{eq:Hn_assump}
 \| H_{n, \tilde \beta_{n,r}}(r)-H_{n, \beta^*}(r) \| \le \|\tilde
\beta_{n,r}-\beta^*\| \sup_{x,y}\|h_{x,y}'(\bar B_{x,y,n} )\| \bigg| \sum_{x,y\in
  \PP_n}\mathds{1}_{\{0<\|x-y\|\le r\}} e_n(x,y)\bigg| .
\end{equation}
On the right hand side, the first factor is bounded by $\|\hat
\beta_n-\beta^*\|$ and hence converges to zero in
probability. The second is bounded in probability since eventually $\|
\bar \beta_{i,x,y,n} - \beta^*\| \le \|
\hat \beta_n - \beta^*\|  \le \eps$ with high probability. The last factor is bounded by $c_2^{2}\hat K_{n}(r)$, where $c_2$ is the constant from \eqref{eq:rhoassumption}, and this is bounded in
probability by the consistency assumption on $\hat K_{n}(r)$.

The convergence of $H_{n,\tilde
	\beta_{n,r}}(r)- H_{n,\hat \beta_n}(r)$ follows from the
      previous result, the analogous result with $\tilde \beta_{n,r}$ replaced by $\hat \beta_n$, and the decomposition
\[ H_{n,\tilde \beta_{n,r}}(r)-H_{n,\hat \beta_{n,r}}(r)= H_{n,\tilde
	\beta_{n,r}}(r)- H_{n,\beta^*}(r)+H_{n, \beta^*}(r)-H_{n,\hat
	\beta_{n,r}}(r).\]
By application of the Campbell formula,
\[ \|\bar H_{n}(r)\| \le \int_{W_n^2} \frac{\mathds{1}_{\{0<\|x-y\|\le
    r\}}}{\rho_{\beta^*}(x)\rho_{\beta^*}(y)|W_{n,x-y}|}g(x-y) \left\|\frac{\dd}{\dd \beta^\T}(\rho_{\beta}(x)\rho_{\beta}(y))\Big|_{\beta=\beta^*}\right\|
\dd x \dd y \le c K(r) \]
for some $c>0$.
That $\Var 	H_{n,\beta^*}(r)$ is
   $O(n^{-1})$ follows from Lemma~1 in \cite{waagepetersen:guan:09}.
\end{proof}

\begin{cor}
Assume that \eqref{eq:consistency}, \eqref{eq:rhoassumption} and \eqref{eq:fastdecay} hold.
Then $\hat K_{n,\hat
		\beta_n}(r)$ is consistent.
\end{cor}
\begin{proof}
Combining \eqref{eq:KTaylor}
with Proposition~\ref{prop:H_est}, we have
\begin{align*}
&\hat K_{n,\hat \beta_n}-K(r) = H_{n,\hat \beta_{n,r}}(r) (\hat
  \beta_n-\beta^*)+ \hat K_{n, \beta^*}(r)-K(r) +o_P(1) \\
= & H_{n,\beta^*}(r) (\hat
  \beta_n-\beta^*)+ \hat K_{n}(r)-K(r) +o_P(1)\\
= &\bar H_{n}(r) (\hat
  \beta_n-\beta^*)+ (H_{n, \beta^*}(r) - \bar H_{n})(\hat
  \beta_n-\beta^*) +\hat K_{n}(r)-K(r) +o_P(1)
=  o_P(1)
\end{align*}
by consistency of $(\hat \beta_n^\T,\hat K_{n}(r))$, boundedness of $\bar H_{n}(r)$,  and since
$H_{n,\beta^*}(r)-\bar H_{n}$ is bounded in probability. 
\end{proof}

\subsection{Functional convergence}\label{sec:functional}
Fix $0\le r_0 < R <\infty$. We will prove functional convergence of
the process $\{\sqrt{n}(\hat{K}_{n,\hat
  \beta_n}(r)-K(r))\}_{r\in[r_0,R]}$ with estimated intensity under
the assumption that functional convergence holds for the process
$\{\sqrt{n}(\hat{K}_{n}(r)-K(r))\}_{r\in [r_0,R]}$ with known intensity. 

Define $\bar H_{n}(r)=\E H_{n,\beta^*}(r)$ as in
Proposition~\ref{prop:H_est}.
We make the following further assumptions:
\begin{enumerate}[label=(\roman*)]
\item \label{itemi} The matrices $C_n$ converge to a fixed invertible matrix
$C$. 
 
 \item \label{itemii} The matrices $\bar H_{n}(r)$ converge uniformly to a function $H(r)$.

\item \label{itemiii} The process
  $\{\sqrt{n}(\hat{K}_{n}(r)-K(r))\}_{r\in[r_0,R]}$ converges in
  Skorokhod topology to a Gaussian process with a limiting covariance
  function $c(s,t)$, $s,t \ge 0$.
\end{enumerate}

Note that by Proposition~\ref{prop:H_est}, \ref{itemii} together with \eqref{eq:consistency} and \eqref{eq:rhoassumption}
implies that $A_n \to A$ in probability where $A$ is defined as $A_n$,
but with $H_n(\tilde B_n)$ replaced by the matrix with rows
  $H(r_i)$, $i=1,\ldots,k$. Combining further with \eqref{eq:clt} we obtain
\begin{equation} \label{eq:obs}
\sqrt{n}V_n \rightarrow N(0,A  \Sigma A^T) 
\end{equation}
where $\Sigma=C^{-1}(C^{-1})^{\T}$.

Our main result is the following functional central limit theorem. The proof is postponed to Section~\ref{sec:proof}.
\begin{thm}\label{tight}
	Suppose that \eqref{eq:consistency}, \eqref{eq:rhoassumption}, \eqref{eq:fastdecay}, \eqref{eq:boundedg},
	 \eqref{eq:clt}, and \ref{itemi}-\ref{itemiii}
        hold. Then $\big\{\sqrt{n}\big(\hat{K}_{n,\hat
          \beta_n}(r)-K(r)\big)\big\}_{r\in[r_0,R]}$ converges in
        Skorokhod topology to a Gaussian process with limiting covariance
        function given by
\begin{equation}\label{eq:lim_cov}
 \tilde{c}(s,t)=H(s) \Sigma_{11} H(t)^\T+H(s) \Sigma_{2,t}+H(t) \Sigma_{2,s} +c(s,t),
 \end{equation}
 where $\Sigma_{11}$ is the limiting covariance matrix of
 $\sqrt{n}(\hat \beta_{n}-\beta^*)$ and $\Sigma_{2,r}$ is the  limiting
 covariance vector between $\sqrt{n}(\hat \beta_{n}-\beta^*)$ and
 $\sqrt{n}\hat K_{n}(r)$, $r \in [r_0,R]$.
 \end{thm}


\section{Point processes satisfying the assumptions}\label{sec:pp}

The most complete characterization of asymptotic normality
 is obtained in the case of constant intensity, see
Section~\ref{sec:constantintensity}. We consider further the case of
the commonly used \cite[e.g.\ ][]{Baddeley:Rubak:Wolf:15} log-linear model for the intensity function in Section~\ref{sec:loglinear}.

\subsection{Constant intensity}\label{sec:constantintensity}
		In the case of constant intensity $\rho_\beta(x) =
                \beta$,  the standard estimator for $\beta$ is
                $\hat{\beta}_n = \#(\PP_n)/|W_n|$, where $\#$ denotes
                cardinality, see e.g.\ \cite{chiu}. Functional
                convergence with known intensity was
                established in \cite{biscio:svane:22} for the class of stationary conditionally $m$-dependent point processes having exponential decay of correlations  and satisfying two extra conditions \cite[Cond.\ {\bf (M)}]{biscio:svane:22} and \cite[Cond.\ {\bf (R)}]{biscio:svane:22}. Exponential decay of correlations means that \eqref{eq:fastdecay} holds for all $p$ and $q$ with $\phi_{p,q}$ being an exponential function, see \cite[Sec.\ 1.1]{yogesh} for the precise definition.
		
		Assumption \eqref{eq:clt} can be verified assuming
                only exponential decay of correlations and
                \cite[Cond.\ {\bf (M)}]{biscio:svane:22}, and hence
                applies to all point processes discussed in
                \cite[Sec.\ 2.2.2]{yogesh} (and also \cite[Sec.\
                2.2.1]{yogesh}). Assuming additionally  conditional $m$-dependence, \ref{itemi} can be verified when replacing \cite[Cond.\ {\bf (R)}]{biscio:svane:22} by the following:
		For any $r_0\le r_1<r_2 \le R$, define events 
		\begin{align*}
		F_{1} ={}& \lbrace  \PP_{(5\tilde{R})^{d}} = \emptyset  \rbrace \\
		F_{2} ={}& \left\lbrace \forall x,y \in  \PP_{(5\tilde{R})^{d}}: x=y\vee \|x-y\|> r_1 \right \rbrace \cap \left\lbrace  \exists x,y \in \PP_{(3\tilde{R})^{d}}: 0<\|x-y\|\le r_2 \right\rbrace\\
		F_{3} ={}& \left\lbrace  \PP_{(5\tilde{R})^{d}}\backslash \PP_{(3\tilde{R})^{d}} = \emptyset \right \rbrace  \cap \left\lbrace  \# (\PP_{(3\tilde{R})^{d}}) = 1 \right\rbrace
		\end{align*}
		where $\tilde R = \max\{m,R\}$.
		Then we require
		\begin{enumerate}
			\item[{\bf (R1)}] 
			\begin{equation*}
			\E\big[\min_{i \in \{1,2\}} P\big( \PP_{(5\tilde{R})^d} \in F_i\,|\, \sigma(\Lambda , \PP \setminus W_{\tilde{R}^d})\big)  \big] > 0,
			\end{equation*}
			\item[{\bf (R2)}] 
			\begin{equation*}
			\E\big[\min_{i \in \{1,3\}} P\big( \PP_{(5\tilde{R})^d} \in F_i\,|\, \sigma(\Lambda , \PP \setminus W_{\tilde{R}^d})\big)  \big] > 0,
			\end{equation*}
		\end{enumerate}
		where $\Lambda$ is the random measure from the definition of conditional $m$-dependence. Examples of conditionally $m$-dependent point processes satisfying all assumptions are log-Gaussian Cox processes with exponentially decaying covariance function  and Mat\'{e}rn cluster processes, see \cite[Sec.\ 4]{bchs20}.  
		
		Condition \eqref{eq:clt}  can also be shown for the class of Gibbs processes considered in \cite{biscio:svane:22}. Recent developments \cite[Thm.\ 3]{otto} (or \cite[Thm.\ 4.11]{benes}) allow a generalization to Gibbs processes satisfying  \cite[Cond.\ {\bf (A)}]{otto}, which is more general than the assumption in \cite{biscio:svane:22}, c.f.\ the discussion in \cite[Rem.\  3.5]{benes}. Positive definiteness of $\Sigma$ can be obtained as in \cite[Prop.\ 6.2]{biscio:svane:22} if the process satisfies {\bf (R1)} and {\bf (R2)} with $\tilde{R}$ larger than the interaction radius and $\Lambda$ trivial. Verifying this is often straightforward for a given Papangelou intensity.
			
The functional convergence \ref{itemiii} established for the point
  processes considered in \cite{biscio:svane:22} generalizes to Gibbs processes satisfying \cite[Cond.\ {\bf (A)}]{otto} since fast decay of correlations can be derived from \cite[Thm.\ 3.4]{benes}, and condition {\bf (M)} can be derived from \cite[Lem.\ 2.6]{benes}.
Condition \ref{itemii} is easily verified, noting that 
\begin{equation} \label{eq:H}
H_{n,\beta}(r) = -2\beta^{-1} \hat{K}_{n,\beta}(r), \qquad H(r) =-2(\beta^*)^{-1} {K}(r).
\end{equation}
The condition \ref{itemi} follows from conditional $m$-dependence and
  {\bf (R1)} and {\bf (R2)}. We then obtain the following corollary.
\begin{cor}\label{cor:constant}
	For all point processes considered in \cite{biscio:svane:22} and  Gibbs processes satisfying \cite[Cond.\ {\bf (A)}]{otto}, if additionally {\bf (R1)} and {\bf (R2)} are satisfied, then $\big\{\sqrt{n}\big(\hat{K}_{n,\hat \beta_n}(r)-K(r)\big)\big\}_{r\in[r_0,R]}$ converges in Skorokhod topology to a centered Gaussian process with covariance structure given by \eqref{eq:lim_cov}.
\end{cor}

For any point process with fast decay of correlations, the limiting
covariance with known intensity $\Sigma$ can be computed from
Campbell's formulae. The computations are omitted, but are very similar to those in Appendix \ref{sec:convergence}. This yields 
\begin{align}\nonumber
\lim_{n\to \infty} n \Var(\hat \beta_n) =& \beta^2\int_{\R^{d}} (g(x) - 1) \d x + \beta,\\ \nonumber
\lim_{n\to \infty} n \Cov(\hat \beta_n, \hat K_n (r)) =& \beta\int_{\R^{2d}}(g^{(3)}(o,x,y)- g(x) )\one_{\{\|x\|\le r\}} \d x  \d y + 2 K(r),\\
\nonumber
\lim_{n\to \infty} n \Cov ( \hat{K}_{n}(r_1), \hat{K}_{n}(r_2)) =&\int_{\R^{3d}} \left(g^{(4)}(o,x,y,z) - g(x)g(y-z)\right)\one_{\{\|x\|\le r_1,\|y-z\|\le r_2\}} \d x  \d y \d z\\ \nonumber
&+ \frac{4}{ \beta}\int_{\R^{2d}} g^{(3)}(o,x,y)  \one_{\{\|x\|\le r_1,\|y\|\le r_2\}} \d x  \d y\\ \label{eq:cov_known}
&+ \frac{2}{ \beta^2}\int_{\R^d} g(x)  \one_{\{\|x\|\le r_1\wedge r_2\} } \d x.
\end{align}
 Note that
all integrals converge due to the fast decay of correlations assumption.

With unknown intensity, the limiting covariance structure is \eqref{eq:lim_cov}.
 Using this, we obtain from \eqref{eq:cov_known},
\begin{align}\nonumber
&\lim_{n\to \infty} n \Cov ( \hat{K}_{n,\hat\beta_n}(r_1), \hat{K}_{n,\hat \beta_n}(r_2)) = \lim_{n\to \infty} n^{-1} \Cov ( \hat{K}_{n}(r_1), \hat{K}_{n}(r_2))\\ \nonumber
&- {2} \int_{\R^{2d}} \left(g^{(3)}(o,x,y)-g(x) \right) \left( K(r_1)\one_{\{\|x\|\le r_2\}} +K(r_2) \one_{\{\|x\|\le r_1\}}\right) \d x \d y\\ 
&+{4}K(r_1)K(r_2)\left( \int_{\R^d} (g(x) - 1) \d x - \beta^{-1}\right). \label{cov_estimated}
\end{align}

For Poisson processes in $\R^2$, these formulas reduce to the covariance formulas given in \cite{heinrich:91}:
\begin{align} \nonumber
&\lim_{n\to \infty} n \Cov ( \hat{K}_{n}(r_1), \hat{K}_{n}(r_2)) = 2\pi (r_1\wedge r_2)^2/\rho^2 + 4\pi^2r_1^2 r_2^2/ \rho,\\ 
&\lim_{n\to \infty} n \Cov ( \hat{K}_{n,\hat\beta_n}(r_1), \hat{K}_{n,\hat \beta_n}(r_2)) = 2\pi (r_1\wedge r_2)^2/\rho^2. \label{eq:poisson_cov}
\end{align}
For general point processes, the explicit covariance formulas are
 more complicated but may be evaluated numerically after plugging in
 estimates of the normalized joint intensities. For instance, in case
 of log-Gaussian Cox processes and Neyman-Scott processes, explicit
 parametric expresssions for the normalized joint intensities are
 available \citep{moeller:syversveen:waagepetersen:98,jalilian:16} enabling parametric estimation of these.

Note that by \eqref{eq:poisson_cov}, in case of a
	Poisson process, it is actually beneficial to
estimate the intensity since the asymptotic variance is smaller in the
unknown intensity case. We expect this to hold for more general point
processes. For example, for point processes satisfying $g^{(3)}(0,x,y)
\ge g(x)$ for all $x,y\in \R^d$, the middle term in
\eqref{cov_estimated} gives a negative contribution to the
variance. This holds for instance for Poisson processes, shot noise
Cox processes and log-Gaussian Cox processes with non-negative
covariance, see formulas for $\rho^{(k)}$ in \cite{CMW}. Moreover, if
$g$ is sufficiently close to $1$, the last term in
\eqref{cov_estimated} is also negative, {reducing further the} point-wise variance when the intensity is estimated.

\subsection{Log-linear intensity}\label{sec:loglinear}

The log-linear model $\rho_\beta(u)=\exp(z(u)^\T \beta)$ is the
default model when covariates $z(u)$, $u \in W_n,$ are available for
explaining variation in the intensity function. We consider here the
case where $\hat \beta_n$ is the first order composite likelihood
estimate 
obtained by maximizing the likelihood function of a Poisson process
with intensity function $\rho_\beta$
\cite[see][for theoretical and practical details]{schoenberg:05,waagepetersen:07,guan:loh:07,waagepetersen:guan:09,choiruddin:coeurjolly:waagepetersen:21}. 
Following for example \cite{waagepetersen:guan:09}, under certain conditions,
\[ |W_n|\bar S_n (\hat \beta_n - \beta^*) = e_n(\beta^*)+ o_{P}(1) \]
where $\bar S_n$ is the normalized sensitivity
\begin{equation}\label{eq:sensitivity}
 \bar S_n = \frac{1}{|W_n|}\int_{W_n} z(u)z(u)^\T \exp(z(u)^\T \beta^*) \dd
    u
    \end{equation}
and
\begin{equation}\label{eq:en}
e_n(\beta^*)=\sum_{u \in \PP_n \cap W_n} z(u) - \int_{W_n}z(u)
\rho_{\beta^*}(u) \dd u 
\end{equation}
is the Poisson composite likelihood score.
Let $\Delta K= (\hat K_n(r_i)-K(r_i))_{i=1}^k$ and
\[\Sigma_n=\Var((|W_n|^{-1/2}e_n(\beta^*)^\T,|W_n|^{1/2}\Delta K^\T)^\T).\]
Then, e.g.\ assuming $\alpha$-mixing \cite{waagepetersen:guan:09}, 
\[ \Sigma_{n}^{-1/2} (|W_n|^{-1/2}e_n(\beta^*),|W_n|^{1/2} \Delta
  K^\T)^\T \rightarrow N(0,I_{p+k}).\] Hence, letting $B_n$ be block diagonal with blocks $\bar S_n$
and $I_k$, 
$ |W_n|^{1/2}(\Sigma_n/|W_n|)^{-1/2}  B_n U_n$
converges to $N(0,I_{p+k})$. We can thus take $C_n=  (\Sigma_n/|W_n|)^{-1/2} B_n$.

 The validity of assumption~\ref{itemi}
depends on the behaviour of $z$ on the infinite domain
$\R^d$. The normalized sensitivity $\bar S_n$ is for example obviously a spatial average
that has a limiting value $\bar S$ when $z$ is a realization of an
ergodic process. In this situation, we can also show (Appendix~\ref{sec:convergence}) under
\eqref{eq:decayg}-\eqref{eq:decayg4}  that $\bar H_n$ and $\Sigma_n$ have limits $H$ and
$\Sigma$. Hence $C_n$ has a limit $C$ too. Specifically, $\bar{H}_{n,\beta^*}(r)$ converges to $-2 K(r)
  \bar z$ where $\bar z= \lim_{n \rightarrow \infty}
  {|W_n|}^{-1}\int_{W_n}z(v) \dd v$ and \ref{itemii} is satisfied.  {The invertibility of $C$ remains an assumption.} 
 
Having established \eqref{eq:clt} and \ref{itemi}, \ref{itemiii} holds since the 
tightness proof from
\cite{biscio:svane:22} goes through in the case of inhomogeneous point
processes with exponential decay of correlations satisfying
\cite[Cond.\ {\bf (M)}]{biscio:svane:22} and intensity function satisfying \eqref{eq:rhoassumption}. Indeed,  one must show Lemma 6.2, Lemma 7.2 and Lemma 7.3 of \cite{biscio:svane:22}. Lemmas 6.2 and Lemma 7.2 are proved in the exact same way using the lower bound on $\rho_{\beta^*}$. Lemma 7.3 also follows in the same way by noting that the proof of \cite[Thm. 1.11]{yogesh} carries over to the case of non-stationary processes. 

\section{Simulation study}

To emphasize the importance of taking into account the effect of
estimating the intensity, we conduct a small scale simulation study
for a Poisson process and a Mat{\'e}rn cluster
process on a sequence of increasing
square observation windows of sidelengths 1, 2, 4 and 8. We use an
intensity of 200 for both types of point processes. As in
\cite{biscio:svane:22} we consider a Kolmogorov-Smirnov goodness-of-fit test
statistic $\sup_{r \in [0,R]}|\hat K_{n,\hat \beta_n}(r)- \pi r^2|$
for the null hypothesis of a homogeneous Poisson process. We
  choose $R=0.05$ to reflect a reasonable upper lag on the unit square
  window. For the Mat{\'e}rn process we use a parent intensity of 25, on
  average 8 offspring for each parent, and a uniform dispersal density on a disc of radius 0.2.

Table~\ref{simstudy} reports rejection probabilities when
  the Kolmogorov-Smirnov test is rejected on the nominal 5\%
  level. The rejection probabilities are computed over 10000
  simulations where for each simulation, the critical value of the
  test is determined from the asymptotic distribution of
  $\sqrt{n}(\hat K_{n,\hat \beta_n}- \pi r^2)$ under the null distribution. Here we consider both
  results obtained with the asymptotic variance formula for estimated
  intensity and the
  asymptotic variance formula pretending the estimated intensity is
  the true intensity (see
  \eqref{eq:poisson_cov}). In both asymptotic variance formulas, the
  unknown intensity is replaced by the estimated
  intensity. In case of the Poisson process, the actual levels of
  the test are close to the nominal level for all window sizes when
  the correct asymptotic variance is used. However, assuming
  erroneously known intensity, the rejection probabilities are far too
  small, completely invalidating the goodness-of-fit test. For the
  Mat{\'e}rn process, using the variance formula for known intensity leads
  to a loss of power of the goodness of fit test.
\begin{table}
\begin{tabular}{l|c|c|c|c}
Window side length & 1 & 2 & 4 & 8 \\ \hline
Poisson (assuming estm.\ intensity)  & 0.053 & 0.053 & 0.052 & 0.051\\
Poisson (assuming known intensity)  & 0.0015 & 0.0011 & 0.0008 & 0.0008 \\\hline
Mat{\'e}rn (assuming estm.\ intensity)  & 0.63 & 1.00 & 1.00 & 1.00\\
Mat{\'e}rn (assuming known intensity)  & 0.31 & 0.96 & 1.00 & 1.00 
\end{tabular}
\caption{Rejection probabilities for Kolmogorov-Smirnov test for
  Poisson process (upper two rows) and Mat{\'e}rn  process (lower two rows)
  on increasing observation windows. For each type of process, the
  first row is using 
  asymptotic variance for estimated intensity and the second row is
  using asymptotic variance
  assuming known intensity.}\label{simstudy}
\end{table}
\section{Proof of Theorem \ref{tight}}\label{sec:proof}
We will need  the following  uniform version of Proposition \ref{prop:H_est}.
	\begin{lem}\label{lem:Hunif}
		Under the assumptions \eqref{eq:consistency}, \eqref{eq:rhoassumption}, \eqref{eq:fastdecay}, \eqref{eq:boundedg}, 
		$\|H_{n,\beta^*}-\bar H_n \|_\infty$ goes to zero in probability. Moreover, $\bar H_n$ is continuous, and so is $H$ if \ref{itemii} is satisfied.
	\end{lem}
	\begin{proof}[Proof of Lemma \ref{lem:Hunif}]
		Let $\delta>0$ be given and let $\eta>0$ be a parameter to be chosen later. Choose $r_0 < r_1 < ... <r_k=R$ such that $|r_i-r_{i+1}|\leq \eta $ and $k\eta \le 2(R-r_0)$. Then
		\begin{align} \label{term1}
                  P\Big(\sup_{r\in [r_0,R]} \|H_{n,\beta^*}(r) -  \bar H_n(r) \| \ge \delta \Big) \le &  P\Big( \sup_{i=0,\ldots,k-1}\sup_{r\in [r_i,r_{i+1}]}  \| H_{n,\beta^*}(r)  - H_{n,\beta^*}(r_i) \| \ge \delta/3 \Big) \\ \label{term2}
                  &+ P\Big(\sup_{i=0,\ldots,k} \|H_{n,\beta^*}(r_i) -  \bar H_n(r_i) \| \ge  \delta/3 \Big)\\ \label{term3}
		 &+  P\Big(\sup_{i=0,\ldots,k-1}\sup_{{r\in [r_i,r_{i+1}]}} \|\bar H_{n}(r_i) - \bar H_{n}(r)  \| \ge  \delta/3  \Big). 
		\end{align}
        	{With $r_i<r$, }
        \begin{equation}\label{eq:something}
        H_{n,\beta^*}(r) - H_{n,\beta^*}(r_i)  =  -\sum_{x,y\in \PP_n} \frac{\one_{A(r,r_i)}(x-y)}{|W_n\cap W_{n,x-y}| \rho_{\beta^*} (x)^2  \rho_{\beta^*} (y)^2}\frac{\d}{\d \beta}( \rho_{\beta^*} (x)  \rho_{\beta^*} (y)) 
        \end{equation}
        where $A(r,t) $ for $r<t$ denotes the annulus $B_t(0) \backslash B_r(0)$ with volume $|A(r,t)|\leq C_1 |t-r|$ where $C_1$ is independent of $t$ as long as $t\leq R$. 
        It follows that
				\begin{equation} \label{eq:varHdiff}
		\sup_{r\in  [r_i,r_{i+1}]} \|H_{n,\beta^*}(r) - H_{n,\beta^*}(r_i) \| \leq  \sum_{x,y\in \PP_n} \frac{\one_{A(r_i,r_{i+1})}(x-y)}{|W_n\cap W_{n,x-y}| \rho_{\beta^*} (x)^2  \rho_{\beta^*} (y)^2}\left\|\frac{\d}{\d \beta}( \rho_{\beta^*} (x)  \rho_{\beta^*} (y)) \right\|.
		\end{equation}
		Let $X_{i,n}$ denote the right hand side. Campbell's formula shows that  $ \E X_{i,n}\le C_2 \eta$, where $C_2$ is independent of $n$ and $r_i$.  Moreover, the computation in Appendix \ref{sec_Hvar} shows that $\Var(X_{i,n}) \le C_3\eta n^{-1}$. Choose $\eta <\delta/(6C_32$. Then by Chebyshev's inequality, 
		$$P\Big(\sup_{r\in [r_i,r_{i+1}]} \|  H_{n,\beta^* }(r) -  H_{n,\beta^* }(r_i) \| \ge \delta/3  \Big) \le P(|X_{i,n}| \ge \delta/3)  \le P(|X_{i,n}-\E X_{i,n}| \ge \delta/6 ) \le \frac{36 C_3 \eta }{n \delta^2},$$
		so \eqref{term1} is bounded by $36 k
                C_3\eta/(n\delta^2)$, which tends to zero as $n\to
                \infty$  since $k\eta$ is bounded.

Proposition \ref{prop:H_est} shows that $\Var H_{n,\beta^*}(r) $ is $O(n^{-1})$, and it is easily verified (by computations similar to Appendix \ref{sec_Hvar}) that the upper bound is uniform in $r$ for $r\in [r_0,R]$. By Chebyshev's inequality, there is a $C_4>0$ such that
		\begin{equation*}
		P\Big(\sup_{i=0,\ldots,k} \|H_{n,\beta^*}(r_i) -  \bar H_n(r_i) \| \geq  \delta/3\Big) \leq \frac{9(k+1)C_4}{n \delta^2},
		\end{equation*}	
		so \eqref{term2} goes to zero for $n\to 0$ for any
                fixed $\eta$.
		
 Taking expectations and applying Campbell's formula in \eqref{eq:something}, we get
		\begin{align} \nonumber
		&\sup_{r\in [r_i,r_{i+1}]} \| \bar H_{n }(r_i) - \bar H_{n }(r) \| \\ \nonumber
		&\le  \int_{W_n^2}  \frac{\one_{A(r_i,r_{i+1})}(u) }{|W_n\cap W_{n,u}| \rho_{\beta^*} (x) \rho_{\beta^*} (x-u)}\left\|\frac{\d}{\d \beta}( \rho_{\beta^*} (x)  \rho_{\beta^*} (x-u)) \right\| g(u) \d x\d u\\ 
		&\le C_5\eta \label{eq:Hcont}
		\end{align}
		for some $C_5$ independent of $r_i$ and $n$. Thus, choosing  $\eta < \delta/(3C_5)$, \eqref{term3} vanishes.
				
		Finally, \eqref{eq:Hcont} shows that $\bar H_n$ is continuous and if uniform convergence in \ref{itemii} is satisfied, $H$ must also be continuous.
		
	\end{proof}

The proof of Theorem \ref{tight} uses some definitions of Skorokhod space, which we briefly recall here, see \cite[Sec. 12]{billingsley} for details.
The Skorokhod space $D[r_0,R]$ of cadlag functions on $[r_0,R]$ is a separable metric space with metric $\mu$ given by
\begin{equation*}
\mu(f_1,f_2) = \inf_{\lambda} \{|\lambda - I|_{\infty} \vee |f_1-f_2\circ \lambda|_{\infty}\},
\end{equation*}
where the infimum runs over all strictly increasing, continuous
bijections $\lambda: [r_0,R] \to [r_0,R]$, $I$ is the identity map,
and $|\cdot|_\infty$ is the sup norm. In the rest of this
  section, unless other mentioned, functions will be restricted to the
  domain $[r_0,R]$, i.e.\ $|f|_{\infty}= \sup_{r \in [r_0,R]}|f(r)|$.

The tightness condition we apply makes use of the cadlag modulus of
continuity $\omega_{f}'(\delta)$ defined by 
\begin{equation}\label{eq:def_omega}
\omega_{f}'(\delta) = \inf_{\substack{t_1<\dotsm<t_k\\ |t_i - t_{i-1}|>\delta}}\max_{i=1,\ldots,k} \sup_{s,t\in [t_{i-1},t_i)} |f(s)-f(t)|. 
\end{equation}
We will need the following property: If $f_2$ is continuous on
$[r_0,R]$, then it is uniformly continuous. Hence there is a function
$g_{f_2}(\delta)$ with $\lim_{\delta\to 0}g_{f_2}(\delta) = 0$ such
that $|s-t|\le \delta$ implies $|f_2(s)-f_2(t)|\le
g_{f_2}(\delta)$. Since it is enough to take the infimum in
\eqref{eq:def_omega} over all partitions with $|t_i-t_{i-1}|\le
2\delta$, we get for any cadlag $f_1$
\begin{equation}\label{eq:sum_omega}
\omega_{f_1+f_2}'(\delta) \le   \omega_{f_1}'(\delta) + g_{f_2}(2\delta).
\end{equation}

\begin{proof}[Proof of Theorem \ref{tight}]
	
	Recall the decomposition
	\[ \hat K_{n,\hat \beta_n}-K(r) = H_{n,\tilde \beta_{n,r}}(r) (\hat
	\beta_n-\beta^*)+ \hat K_{n}(r)-K(r). \]
	The proof relies on the observation that $H_{n,\tilde \beta_{n,r}}(r)$ converges to a continuous deterministic function $H(r)$, $\sqrt{n}(\hat
	\beta_n-\beta^*)$ is constant in $r$ and converges in
        distribution, and $\sqrt{n} (\hat K_{n}(r)-K(r))$ converges in
        Skorokhod space. Thus, the proof below can be viewed as a
        stochastic process analogue to Slutsky's theorem.
	
	To simplify,  we  write
	\begin{equation}\label{eq:slutsky}
	\sqrt{n}(\hat K_{n,\hat \beta_n}-K(r) ) = H_{n,\tilde \beta_{n,r}}(r)Y_n +  \sqrt{n}(\hat K_{n}(r)- K(r)) =(H_{n,\tilde \beta_{n,r}}(r)-{H}(r) ) Y_n + Z_n(r)
	\end{equation}
	where $Y_n = \sqrt{n} (\hat \beta_n-\beta^*)$ and converges in distribution to a Gaussian vector $Y$ by assumption {\eqref{eq:obs}}, and 
	\begin{equation*}
	Z_n(r) = {H}(r) Y_n + \sqrt{n}(\hat K_{n}(r) - K(r)).
	\end{equation*}

	By \cite[Thm. 3.1]{billingsley}, in order to show functional convergence, 
	it is enough to show:
	\begin{itemize}
		\item[a.] Convergence of $\mu\big(Z_n,\sqrt{n}\big(\hat K_{n,\hat \beta_n}-K) \big) \big) \to 0$ in probability.
		\item[b.] Convergence of $Z_n$ in distribution in Skorokhod topology 
		to a Gaussian process with the covariance structure given by \eqref{eq:lim_cov}.
	\end{itemize}
	
	To show a., note that $\mu(f_1,f_2)\le |f_1-f_2|_\infty$, so
	\begin{equation*}
	\mu\big(Z_n, \sqrt{n}\big(\hat K_{n,\hat \beta_n}-K \big)\big)\le \big| \big(H_{n,\tilde \beta_{n,r}}- H \big) Y_n \big|_\infty \leq \big\| H_{n,\tilde \beta_{n,r}}-  H\big\|_\infty \|Y_n\|.
	\end{equation*}
{	Since $Y_n\to Y$ in distribution, $\limsup_n P(\|Y_n\|\ge  M) \le P(\|Y\|\ge
	M)$ and hence $\|Y_n\|$ is bounded in probability.
	 It remains to show that  $\|H_{n,\tilde \beta_{n,r}}- H\|_\infty$ goes to zero in probability. We write
	 	\begin{equation}\label{eq:oldiii}
	 \|H_{n, \tilde{\beta}_{n,r}}-  H\|_\infty \le \|H_{n, \tilde{\beta}_{n,r}}-H_{n,\beta^*} \|_\infty  + \|H_{n,\beta^*}-\bar H_n \|_\infty +  \|\bar H_{n}- H \|_\infty .
	 	\end{equation}
		 	 The first term goes to zero in probability because the bound in \eqref{eq:Hn_assump} is uniform in $r$ since $\hat{K}_n(r) \le \hat K_n(R)$, the middle term goes to zero in probability by Lemma \ref{lem:Hunif}, and the third term goes to zero by \ref{itemii}. 
 }

	To show b., note that convergence of finite dimensional distributions follows from the observation \eqref{eq:obs}. It remains to show tightness. According to \cite[Thm. 13.2]{billingsley}, tightness of a sequence $Z_n$ is equivalent to the following two conditions:
	
	\begin{itemize}
		\item[1.] $\lim_{a\to \infty}\limsup_n P(|Z_n|_{\infty}\ge a) = 0$.
		\item[2.] For any $\eps>0$: $\lim_{\delta\to 0} \limsup_n P(\omega_{Z_n}'(\delta)\ge \eps) = 0$.
	\end{itemize}	
	
	To show 1., note
	\begin{align*}
	&P(|H Y_n +  \sqrt{n}(\hat K_n- K) |_{\infty} \ge a) \\
	&\le P(|H Y_n|_{\infty} \ge {a}/{2}) + P(|\sqrt{n}(\hat K_n- K)|_{\infty} \ge {a}/{2}).
	\end{align*}
	Taking $\lim_{a\to \infty}\limsup_n$, the latter term vanishes by tightness of $\sqrt{n}(\hat K_n(r)- K(r))$. 		
	The first term satisfies
	\begin{align*}
	P\left(|H Y_n|_{\infty}\ge a/2\right) \le  P\left(\|Y_n\| \ge \sqrt{a/2}\right) + P\left(\|H\|_\infty \ge \sqrt{a/2}\right).
	\end{align*}
	Clearly, $\lim_{a\to \infty} P(\|H\|_\infty \ge a/2) =0$ since $H$ is continuous and hence bounded on $[r_0,R]$ by Lemma \ref{lem:Hunif}. 
	Moreover, since $Y_n\to Y$  in distribution,
	\begin{equation*}
	\lim_{a\to \infty}\limsup_n P(\|Y_n\| \ge \sqrt{a/2})\le\lim_{a\to \infty}P(\|Y\| \ge \sqrt{a/2})= 0.
	\end{equation*}

	To show 2., we use \ref{itemii} and \eqref{eq:sum_omega} to obtain a ($p$-dimensional) function $g_{H}$ such that
	\begin{equation*}
	\omega_{\sqrt{n} (\hat K_n-K)+H Y_n}'(\delta) \le 	\omega_{ \sqrt{n} (\hat K_n-K)}'(\delta) + \|Y_n\|\|g_{ H}(2\delta)\|.
	\end{equation*}
	It follows that
	\begin{multline*}
	P\left(\omega_{Y_nK+\sqrt{n} (\hat
          K_n-K)}'(\delta)\ge \eps \right) 
	\le P\left(	\omega_{ \sqrt{n} (\hat K_n-K)}'(\delta)\ge \eps/2 \right) + P\left(\|Y_n\| \ge (\eps/2) \|g_{ H}(2\delta)\|^{-1}\right).
	\end{multline*}
	Taking $\lim_{\delta\to 0} \limsup_n$ yields 0 in both terms because of tightness of $\sqrt{n} (\hat K_n-K)$ and convergence in distribution of $Y_n$.

\end{proof}

\bibliography{../lit}

\bibliographystyle{plainnat}

\appendix

\section{Convergence results for the log-linear model}\label{sec:convergence}

Throughout this appendix $z$ is a realization from a stationary, bounded, and ergodic random field on $\R^d$.

\subsection{Convergence of $\bar H_{n}$}
By \eqref{eq:Hn} and Campbells formula,
\begin{align*} 
\bar H_{n}(r) &= - \E \sum_{x\in \PP_n}
\sum_{y\in \PP_n}
\frac{\mathds{1}_{\{0<\|x-y\|\le
    r\}}}{\rho_{\beta^*}(x)\rho_{\beta^*}(y)} \frac{\dd}{\dd \beta^\T}\log(\rho_{\beta}(x)\rho_{\beta}(y))|_{\beta=\beta^*}
e_n(x,y) \\
&=  -\int_{W_{n}^2} \frac{\mathds{1}_{\{0<\|x-y\|\le
    r\}}}{|W_n \cap W_{n,x-y}|}g(x-y) [z(x)+z(y)] \dd x \dd y\\
&= -   2 \int_{B_r(0)} g(u) \frac{1}{|W_n \cap W_{n,u}|}\int_{W_n \cap W_{n,u}} z(v) \dd v \dd u.
\end{align*}
The inner integral is a spatial average of $z(v)$ over
$W_n\cap W_{n,u}$. By the Pointwise Ergodic Theorem \cite[Thm. 10.14]{kallenberg}, this converges to a limiting value $\bar z$ since $z(\cdot)$ is a realization of a stationary ergodic process. It then follows by boundedness of $z$ and dominated convergence that $\bar
H_{n,\beta^*}(r)$ converges to $-2 K(r) \bar z$. Clearly, the convergence is uniform since 
\begin{align*} 
\|\bar H_{n}(r) - H(r) \| \leq   2 \int_{B_R(0)} |g(u)| \left\| \frac{1}{|W_n \cap W_{n,u}|}\int_{W_n \cap W_{n,u}} z(v) \dd v - \bar z \right\|\dd u,
\end{align*}
where the right hand side is independent of $r$.

\subsection{Convergence of $\bar S_n$}
The normalized sensitivity  $\bar S_n$ in \eqref{eq:sensitivity} is obviously a spatial average
that has a limiting value $\bar S$ when $z$ is ergodic.

\subsection{Upper left block of $\Sigma_n$}
The variance of $|W_n|^{-1/2} e_n(\beta^*)$ is by \eqref{eq:en}
\[
\frac{1}{|W_n|} \Var \bigg( \sum_{x \in \PP_n \cap
  W_n} z(x)\bigg) = \bar S_n + \frac{1}{|W_n|}\int_{W_n^2} z(x)z(y)^\T \rho_{\beta^*}(x)\rho_{\beta^*}(y) [g(x-y)-1] \dd x \dd y,
\]
where the first term on the right hand side converges to a limiting
value $\bar S$. 
The double integral can be rewritten as
 \[
  \int_{\R^d}   [g(v)-1] \frac{1}{|W_{n}|}   \int_{W_{n}\cap W_{n,v}} z(u) z(u-v)^\T  \rho_{\beta^*}(u)\rho_{\beta^*}(u-v) \dd u \dd v.
 \] 
 Since ${|W_{n}\cap W_{n,v}|}/{|W_{n}|} \to 1$, the inner integral converges pointwise to a spatial average which is uniformly bounded in $n$ and $v$ by the assumptions on $z$. Dominated convergence and \eqref{eq:decayg} then shows that the limit of the double integral is
  \[
 \Sigma_{11} = \int_{\R^d} [g(v)-1] \lim_{n\to \infty} \bigg(\frac{1}{|W_{n}|}  \int_{W_{n}}  z(u) z(u-v)^\T  \rho_{\beta^*}(u)\rho_{\beta^*}(u-v)\dd u \bigg) \dd v.
 \] 

\subsection{Lower right block of $C_n$}
The lower right block of $\Sigma_n$ has entries $|W_n|\Cov[\hat
K_{n}(r_i),\hat K_n(r_j)]$.  Assume $r_i\le r_j$. Following the proof of Lemma~1 in
\cite{waagepetersen:guan:09} this is the sum 
\begin{align*} 
&2 |W_n| \int_{W_n^2} \frac{\mathds{1}_{\{\|x-y\| \le r_i\}}}{|W_{n,x-y}|^2
      \rho_{\beta^*}(x)\rho_{\beta^*}(y)}g(x-y) \dd x \dd y\\
&+ 4 |W_n|\int_{W_n^3} \frac{\mathds{1}_{\{\|x-y\| \le r_i,\|x-z\| \le
      r_j\}}}{|W_{n,x-y}||W_{n,x-z}|\rho_{\beta^*}(x)}g^{(3)}(x,y,z)
  \dd x \dd y \dd z,\\
&+ |W_n| \int_{W_n^4} \frac{\mathds{1}_{\{\|x-y\| \le r_i,\|z-w\| \le
      r_j\}}}{|W_{n,x-y}||W_{n,z-w}|} [g^{(4)}(x,y,z,w)-g(x-y)g(z-w)]
  \dd u \dd v \dd w \dd z. 
  \end{align*}
 The first term equals 
  \begin{align*} 
  2  \int_{\R^d} g(w)\mathds{1}_{\{\|w\| \le r_i\}}  \frac{|W_n|}{|W_{n,w}|^2} \int_{W_n \cap W_{n,w}}
  	\frac{1}{\rho_{\beta^*}(u)\rho_{\beta^*}(u-w)} \dd u \dd v, 
  \end{align*}
  which converges to
    \begin{align*} 
  2  \int_{\R^d} g(w)\mathds{1}_{\{\|w\| \le r_i\}}  \lim_{n\to \infty}\bigg( \frac{1}{|W_{n}|} \int_{W_n }  \frac{1}{\rho_{\beta^*}(u)\rho_{\beta^*}(u-w)} \dd u\bigg) \dd w.
  \end{align*}
The second term is 
\begin{align*} 
4 \int_{B_{r_i}(0)\times B_{r_j(0)}} g^{(3)}(0,v,w)\frac{|W_n|}{|W_{n,-w_1}||W_{n,-w_2}|}\int_{W_n \cap W_{n,-w_1} \cap W_{n,-w_2} }\frac{1}{\rho_{\beta^*}(u)}
\dd u \dd v \dd w.
\end{align*}
Ergodicity and  dominated convergence yields that the limit exists and is
\begin{align*} 
4 \int_{B_{r_i}(0)\times B_{r_j(0)}} g^{(3)}(0,v,w)\lim_{n\to \infty}\bigg(\frac{1}{|W_{n}|}\int_{W_n} \frac{1}{\rho_{\beta^*}(u)}
\dd u \bigg) \dd v \dd w.
\end{align*}
After change of variables $u_1=x-y$, $u_2=z-w$, $u_3=z$, $u_4=w$, the fourth term is  
\[ \int_{W_n \times B_{r_i}(0)\times B_{r_j}(0)} \frac{|W_n|}{|W_{n,u_1}||W_{n,u_2}|}\int_{W_n}
    [g^{(4)}(0,u_1,u_4,u_4+u_2)-g(u_1)g(u_2)] \dd u_3 \dd u_1 \dd u_2
    \dd u_4, \]
which by \eqref{eq:decayg4} converges to
\[ \int_{B_{r_i}(0) \times B_{r_j}(0)}  \int_{\R^d}
    [g^{(4)}(0,u_1,u_4,u_4+u_2)-g(u_1)g(u_2)]\dd u_4 \dd u_1 \dd u_2
    .\]
Hence the lower right block of $\Sigma_n$ converges to a matrix $\Sigma_{22}$.
\subsection{Lower left block of $\Sigma_n$}
The $ij$th entry in the lower left block of $\Sigma_n$ is $\Cov[ \hat
K_n(r_i),e_n(\beta^*)_j]$. This covariance is 
\begin{align*} 
&\int_{W_n^3} \frac{\mathds{1}_{\{\|x-y\|\le r_i\}}}{|W_{n,x-y}|}z_j(w)
  \rho_{\beta^*}(w)[g^{(3)}(x,y,w)-g(x-y)] \dd x \dd y \dd w\\
&  + 2 \int_{W_n^2}  \frac{\mathds{1}_{\{\|x-y\|\le r_i}\}}{|W_{n,x-y}|}z_j(x) g(x-y) \dd x
  \dd y .
  \end{align*}
The first term equals
\[ \int_{\R^{d}\times B_{r_i}(0) }  [g^{(3)}(0,v_1,v_2)-g(v_1)] \frac{1}{|W_{n,-v_1}|}\int_{W_n\cap W_{n,-v_1}\cap W_{n,-v_2}} z_j(u+v_2)
\rho_{\beta^*}(u+v_2)\dd u \dd v_1 \dd v_2, \] 
which by dominated convergence, ergodicity, and \eqref{eq:decayg3}  converges to 
\[ \int_{\R^{d}\times B_{r_i}(0) }  [g^{(3)}(0,v_1,v_2)-g(v_1)] \lim_{n\to \infty}\bigg(\frac{1}{|W_{n}|}\int_{W_n} z_j(u+v_2) \rho_{\beta^*}(u+v_2)\dd u \bigg)\dd v_1 \dd v_2. \]
The last term is asymptotically equivalent to
\[2 \int_{\R^n} g(v)  \frac{\mathds{1}_{\{\|v\|\le r_i\}}}{|W_{n,v}|} \int_{W_n\cap W_{n,v}}z_j(u)  \dd u
\dd v, \]
which converges to $2K(r_i)\bar z$.
Hence, the lower left block of $\Sigma_n$ converges to a fixed matrix $\Sigma_{21}$.

\section{Variance of \eqref{eq:varHdiff}} \label{sec_Hvar}
By Campbell's formula, $\Var(X_{i,n})$ is
\begin{align*} 
&2  \int_{W_n^2} \frac{\mathds{1}_{A(r_i,r_{i+1})}(x-y)}{|W_n\cap W_{n,x-y}|^2
	\rho_{\beta^*}(x)^3\rho_{\beta^*}(y)^3} g(x-y) \left\|\frac{\d}{\d \beta} (\rho_{\beta^*}(x)\rho_{\beta^*}(y)) \right\|^2\dd x \dd y\\
&+ 4 \int_{W_n^3} \frac{\mathds{1}_{A(r_i,r_{i+1})}(x-y) \mathds{1}_{A(r_i,r_{i+1})}(x-z)}{|W_n\cap W_{n,x-y}|| W_n\cap W_{n,x-z}|\rho_{\beta^*}(x)^3\rho_{\beta^*}(y)\rho_{\beta^*}(z)}g^{(3)}(x,y,z) \\
&\times \left\|\frac{\d}{\d \beta} (\rho_{\beta^*}(x)\rho_{\beta^*}(y)) \right\| \left\|\frac{\d}{\d \beta} (\rho_{\beta^*}(x)\rho_{\beta^*}(z)) \right\|
\dd x \dd y \dd z,\\
&+ \int_{W_n^4} \frac{\mathds{1}_{A(r_i,r_{i+1})}(x-y) \mathds{1}_{A(r_i,r_{i+1})}(z-w) }{|W_n\cap W_{n,x-y}||W_n\cap W_{n,z-w}|\rho_{\beta^*}(x)\rho_{\beta^*}(y)\rho_{\beta^*}(z)\rho_{\beta^*}(w)} \\
& \times [g^{(4)}(x,y,z,w)-g(x-y)g(z-w)] \left\|\frac{\d}{\d \beta} (\rho_{\beta^*}(x)\rho_{\beta^*}(y)) \right\| \left\|\frac{\d}{\d \beta} (\rho_{\beta^*}(w)\rho_{\beta^*}(z)) \right\|
{\dd x \dd y} \dd w \dd z. 
\end{align*}
Using the assumption \eqref{eq:rhoassumption} on $\rho$,  boundedness of $g^{(l)}$ \eqref{eq:boundedg} and \eqref{eq:decayg4}, this is bounded by
\begin{align*} 
&2C  \int_{W_n^2} \frac{\mathds{1}_{A(r_i,r_{i+1})}(x-y)}{|W_n\cap W_{n,x-y}|^2
}  \dd x \dd y\\
&+ 4 C \int_{W_n^3} \frac{\mathds{1}_{A(r_i,r_{i+1})}(x-y) \mathds{1}_{A(r_i,r_{i+1})}(x-z)}{|W_n\cap W_{n,x-y}||W_n\cap W_{n,x-z}| }
\dd x \dd y \dd z,\\
&+ C \int_{W_n^4} \frac{\mathds{1}_{A(r_i,r_{i+1})}(x-y) \mathds{1}_{A(r_i,r_{i+1})}(z-w)}{|W_n\cap W_{n,x-y}||W_n\cap W_{n,z-w}|} [g^{(4)}(x,y,z,w)-g(x-y)g(z-w)] {\dd x \dd y} \dd w \dd z. 
\end{align*}
which is clearly of the right order.

\end{document}